\documentclass{amsart}

\usepackage{amsmath}
\usepackage{amsfonts}
\usepackage{amsopn}
\usepackage{amssymb}
\usepackage{graphicx}

\newtheorem{thm}{Theorem}
\newtheorem*{thm13}{Theorem~1.3}

\newtheorem{lemma}[thm]{Lemma}

\theoremstyle{remark}

\DeclareMathOperator{\supp}{supp}

\DeclareMathOperator{\Lip}{Lip}

\DeclareMathOperator{\diam}{diam}

\DeclareMathOperator{\mass}{mass}
\DeclareMathOperator{\FV}{FV}

\newcommand{\co}{\colon}
\newcommand{\dimAN}{\dim_{\text{AN}}}
\newcommand{\CLip}{C^\text{Lip}}

\title[Erratum: ``Lipschitz connectivity and filling
invariants'']{Erratum: ``Lipschitz connectivity and filling
  invariants in solvable groups and buildings''}
\author{Robert Young}
\date{\today}

\begin{document}
\maketitle
\begin{abstract}
  This note corrects some omissions in section 2 of the paper
  ``Lipschitz connectivity and filling invariants in solvable groups
  and buildings.''
\end{abstract}
\emph{In the course of writing this paper, I changed the statement of
  Theorem 1.3, but did not change the proof.  This erratum presents
  the omitted proofs.  Thanks to Moritz Gruber for noticing the
  omission.}

\emph{This erratum replaces Lemmas 2.6--2.8 and the proof of Theorem 1.3.}

Recall the original theorem:
\begin{thm13}[\cite{HigherSol}]
  Suppose that $Z\subset X$ is a nonempty closed subset with metric
  given by the restriction of the metric of $X$.  Suppose that $X$ is
  a geodesic metric space such that the Assouad-Nagata dimension
  $\dimAN(X)$ of $X$ is finite.  Suppose that one of the following is true:
  \begin{itemize}
  \item $Z$ is Lipschitz $n$--connected.
  \item $X$ is Lipschitz $n$--connected, and if $X_p,p\in P$ are the connected components of
    $X\smallsetminus Z$, then the sets $H_p=\partial X_p$ are
    Lipschitz $n$--connected with uniformly bounded
    implicit constant.
  \end{itemize}
  Then $Z$ is undistorted up to dimension $n+1$.  
\end{thm13}

First, we note that the second condition implies the first condition:
\begin{lemma}
  Suppose that $X$ is Lipschitz $n$--connected and that $Z$ is a
  closed subset of $X$.  Let $X_p,p\in P$ be the connected components
  of $X\smallsetminus Z$ and suppose that the sets $\partial X_p$
  are Lipschitz $n$--connected with uniformly bounded implicit
  constant.  Then $Z$ is Lipschitz $n$--connected.
\end{lemma}
\begin{proof}
  Suppose that $f\co S^n\to Z$ is a Lipschitz map.  We claim that
  there is a Lipschitz map $h\co D^{n+1}\to Z$ that extends $f$ and
  such that $\Lip h\lesssim \Lip f$.  By the Lipschitz connectivity of
  $X$, there is a Lipschitz extension $g\co D^{n+1} \to X$ such that
  $\Lip g\lesssim \Lip f$; if the image of $g$ lies in $Z$, we're
  done.  Otherwise, suppose that $X_p$ is a connected component of
  $X\setminus Z$ and let $K_p=g^{-1}(X_p)$.  Then $g$ sends $\partial
  K_p \to \partial X_p$.  The set $\partial X_p$ is Lipschitz
  $n$-connected, so any Lipschitz map from a closed subset of
  $D^{n+1}$ to $\partial X_p$ can be extended to a Lipschitz map on
  all of $D^{n+1}$ (see \cite[Thm.\ 1.2]{AlmgrenIntegral}, \cite[Thm.\
  2]{JohnsonLindenstraussSchechtman}, \cite[Thm.\ 1.4]{LangSch}).  We
  therefore construct a map $h_p\co K_p\to \partial X_p$ such that
  $h_p$ agrees with $g$ on $\partial K_p$ and $\Lip h_p\lesssim \Lip
  g$.  Then
  $$h(x)=\begin{cases}
    g(x) & \text{ if }g(x)\in Z\\
    h_p(x) & \text{ if }g(x)\in X_p\\
  \end{cases}$$ is an extension of $f$, and $\Lip h\lesssim \Lip f$.
\end{proof}

It thus suffices to prove the theorem in the case that $Z$ is
Lipschitz $n$--connected.  We recall some notation from
Section~2 of \cite{HigherSol}; in the following, $\epsilon>0$ is a
small number that will depend on the cycle to be filled.
\begin{itemize}
\item We cover $X$ by a collection $\mathcal{D}=\mathcal{D}(\epsilon)$
  of open sets $D_k, k\in K$ such that $\diam D_k\gtrsim \epsilon$ for
  all $k$.  This cover can be broken into a ``fine'' portion
  consisting of a cover of $Z$ by $\epsilon$--balls and a ``coarse''
  portion consisting of subsets of $X\setminus Z$ whose diameters are
  roughly proportional to their distance from $Z$.  For each
  $D_k\in \mathcal{D}$, we define a 1-Lipschitz function
  $\tau_k\co X\to \mathcal{R}$ such that $\tau_k\ge \epsilon$ on $D_k$,
  $\diam \supp \tau_k\sim \diam D_k$, and $\supp \tau_k$ intersects
  $Z$ if and only if $D_k$ intersects $Z$.
\item $\Sigma=\Sigma(\epsilon)$ is a QC complex based on the nerve of
  the cover of $X$ by the sets $\supp \tau_k$.  We denote the vertices
  of $\Sigma$ by $v_k$.  Then $\dim \Sigma\le 2 \dimAN {X}+1$, and for
  any $k\in K$, the diameter of any simplex of $\Sigma$ containing
  $v_k$ is comparable to $\diam \supp \tau_k$.
\item $g\co X\to \Sigma$ is a map with
  Lipschitz constant independent of $\epsilon$.  The map $g$ is
  defined by normalizing the $\tau_k$'s to obtain a partition of unity
  $g_k(x)=\tau_k(x)/\bar{\tau}(x)$ where $\bar{\tau}(x)=\sum_i
  \tau_i(x)$, then using the $g_k$'s as coordinate functions.  
  In the proof of Lemma~2.5 in \cite{HigherSol}, we showed that for
  each $k$, we have
  $$\Lip(g_k)\sim \diam(\supp \tau_k)^{-1}.$$
  Since $\diam(\supp \tau_k) \gtrsim \epsilon$, we have
  $\Lip(g_k)\lesssim \epsilon^{-1}$ for all $k$.
\end{itemize}

We can use the connectivity of $Z$ to construct a map
$h\co\Sigma^{(n+1)}\to Z$ as in the proof of Theorem 1.4 of
\cite{LangSch}.
\begin{lemma}
  There is a Lipschitz extension $h\co \Sigma^{(n+1)}\to Z$ with
  Lipschitz constant independent of $\epsilon$ such that
  $d(h(g(z)),z)\lesssim \epsilon$ for every $z\in Z$.
\end{lemma}
\begin{proof}
  For each vertex $v_k$ of $\Sigma$, if $D_k$ intersects $Z$, we
  choose $h(v_k)\in Z\cap D_k$.  Otherwise, we let $h(v_k)\in Z$ be
  such that $d(h(v_k),\supp \tau_k)\le 2 d(Z,\supp \tau_k)$.  If two
  vertices $v_k$ and $v_{k'}$ are connected by an edge $e$, then
  $$\ell(e)\lesssim \diam(\supp \tau_k)+\diam(\supp \tau_{k'}),$$
  and $\supp \tau_k$ intersects $\supp \tau_{k'}$.  Thus
  \begin{align*}
    d(h(v_k),h(v_{k'})&\lesssim d(Z,\supp \tau_k)+\diam(\supp
                        \tau_k)+\diam(\supp \tau_{k'})+d(Z,\supp
                        \tau_{k'})\\ 
                      &\lesssim \ell(e),
  \end{align*}
  so $h$ is Lipschitz on $\Sigma^{(0)}$.  Since $Z$ is Lipschitz
  $n$--connected and $\Sigma$ is a QC complex, we can extend $h$
  to a Lipschitz map on $\Sigma^{(n+1)}$.  Finally, if $z\in Z$, then
  $z\in D_k$ for some $k\in K$, and $g(z)$ is in the star of $v_k$.
  It follows that $d(g(z),v_k)\lesssim \epsilon$, and so 
  $$d(h(g(z)),z)\le \Lip(h) d(g(z),v_k) + d(h(v_k),z)\lesssim \epsilon
  + \diam D_k\lesssim \epsilon$$
  as desired.
\end{proof}

Suppose that $\alpha\in \CLip_m(Z)$ is a $m$--cycle and $m\le n$. In
\cite{HigherSol}, we tried to construct a filling of $\alpha$ in $Z$
by constructing a filling in $X$, sending that filling to $\Sigma$,
then approximating it in $\Sigma^{(n+1)}$ and sending it back to $Z$.
That is, we first use the Lipschitz connectivity of $X$ to construct a
chain $\beta$ in $X$ whose boundary is $\alpha$.  Its push-forward
$g_\sharp(\beta)$ can be approximated by a simplicial chain
$P^0_\beta$ so that $\partial P^0_\beta$ is a simplicial approximation
of $g_\sharp(\alpha)$.  Since $\supp P^0_\beta\subset \Sigma^{(n+1)}$,
the $(m+1)$--chain $h_\sharp(P^0_\beta)$ is a chain in $Z$ and its
boundary is $\epsilon$--close to $\alpha$.  It remains to construct an
annulus between $h_\sharp(\partial P^0_\beta)$ and $\alpha$.  In
\cite{HigherSol}, there were some errors in this construction, and we
will correct those issues here.

Theorem~1.3 will follow from the following lemma:
\begin{lemma}\label{lem:alphaGammaLambda}
  There is a $c_\alpha>0$, depending on the number of simplices in
  $\alpha$ and their Lipschitz constants such that for any
  $\epsilon>0$, there are an $m$--cycle
  $\alpha'=\alpha'(\epsilon)\in C_m(\Sigma(\epsilon))$ and two annuli,
  $\gamma\in \CLip_{m+1}(\Sigma(\epsilon))$ and
  $\lambda\in \CLip_{m+1}(Z)$, such that
  \begin{align*}
    \partial \gamma&=g_\sharp(\alpha)-\alpha'&   \mass \gamma &\lesssim c_\alpha \epsilon\\
    \partial \lambda&=\alpha-h_\sharp(\alpha')&   \mass \lambda &\lesssim c_\alpha \epsilon.
  \end{align*} 
\end{lemma}

\begin{proof}[Proof of Theorem~1.3]
  Given $\gamma$ and $\lambda$, we construct a filling of $\alpha$ by letting
  $P_\beta\in C_{m+1}(\Sigma(\epsilon))$ be a simplicial approximation of
  $g_\sharp(\beta)-\gamma$.  Since $\partial(g_\sharp(\beta)-\gamma)$ is
  already simplicial, we have 
  $$\partial P_\beta=\partial(g_\sharp(\beta)-\gamma)=\alpha'.$$
  Then 
  $$\partial(\lambda+h_\sharp(P_\beta))=\alpha,$$ 
  and
  $$\mass(\lambda+h_\sharp(P_\beta))\lesssim \mass \beta + c_\alpha
  \epsilon.$$
  Letting $\epsilon$ go to $0$, we find that $\FV_Z(\alpha)\lesssim
  \FV_X(\alpha)$ as desired.
\end{proof}

It thus suffices to construct $\alpha'$, $\gamma$, and $\lambda$ as
above.

\begin{proof}[{Proof of Lemma~\ref{lem:alphaGammaLambda}}]
The cycle $\alpha'$ will be based on a subdivision of $\alpha$.  For
any $\delta\in (0,1)$, a Euclidean $m$--simplex can be subdivided into
roughly $\delta^{-m}$ simplices of diameter less than $\delta$, so
there is a $c_\alpha>0$ depending on the number of simplices in
$\alpha$ and their Lipschitz constants such that for any
$\delta\in (0,1)$, we can subdivide $\alpha$ into a sum
$\sum_{i=1}^N \Delta_i$ of simplices where $N\le c_\alpha \delta^{-m}$
and
$$\diam \Delta_i \le \Lip \Delta_i <\delta.$$

Let $L_g=\sup_k \Lip(g_k)\sim \epsilon^{-1}$ and let
$$\delta= \frac{1}{2 (\dim(\Sigma) + 1) L_g}\sim \epsilon.$$  Then
$\alpha=\sum_{i=1}^N \Delta_i$, where
$N\lesssim c_\alpha \epsilon^{-m}$.  

We will construct the simplicial cycle $\alpha'$ by sending each
vertex of each $\Delta_i$ to the nearest vertex of $\Sigma$.  For each
point $z\in Z$, let $k(z)\in K$ be an index that maximizes $g_k(z)$
and let $v(z)=v_{k(z)}$.  We claim that if
$z_{i,0},\dots, z_{i,m}\in Z$ are the vertices of $\Delta_i$, then
$v(z_{i,0}),\dots, v(z_{i,m})$ are the vertices of a simplex of
$\Sigma$ (possibly with duplicates).  Since the $g_k$ form a partition
of unity with bounded multiplicity, we know that
$$g_{k(z_{i,j})}(z_{i,j})\ge \frac{1}{\dim(\Sigma)+1}.$$
If $z\in \Delta_i$, then
\begin{equation}\label{eq:suppTau}
  g_{k(z_{i,j})}(z)\ge \frac{1}{\dim(\Sigma)+1} - L_g
  d(z_{i,j},z)>0,
\end{equation}
so $g_{k(z_{i,j})}(z)>0$ for all $j$, and
$\{v(z_{i,0}),\dots, v(z_{i,m})\}$ is the vertex set of a simplex in
$\Sigma$.  We define $\alpha'$ to be the simplicial cycle
$$\alpha'=\sum_i \langle v(z_{i,0}),\dots, v(z_{i,m})\rangle.$$
This is a sum of at most $N$ simplices, each with diameter on the
order of $\epsilon$, so 
$$\mass \alpha'\lesssim N\epsilon^{m}\lesssim c_\alpha$$

Next, we construct $\gamma$ and $\lambda$.  We construct $\gamma$ from a straight-line
homotopy between $g_\sharp(\alpha)$ and $\alpha'$.  Consider $\Sigma$ as a
subset of the infinite simplex $\Delta^K\subset \ell^2(K)$ with vertex set
$\{v_k\}_{k\in K}$.  Let $\Delta^m$ be
the standard $m$-simplex $\Delta^m=\langle e_0,\dots, e_m\rangle$.
We view $\Delta_i$ as a map $\Delta_i\co \Delta^m\to Z$.  Likewise, we
write $\alpha'=\sum_i \Delta'_i$, where
$\Delta'_i\co \Delta^m\to \Sigma$ is the linear map such that
$\Delta'_i(e_j)=v(z_{i,j})$.

Let $x\in \Delta^m$ and let $z=\Delta_i(x)$.  We claim that $g(z)$ and
$\Delta'_i(x)$ are both contained in the same simplex of $\Sigma$.
For $s\in \Sigma$, let $\supp s$ be the vertex set of the minimal
simplex containing $s$; then, by the definition of $g$,
$$\supp g(z) = \{v_k\mid g_k(z)>0\}$$ 
and
$$\supp \Delta'_i(x) = \{v_{k(z_{i,0})},\dots, v_{k(z_{i,m})} \}.$$ 
By \eqref{eq:suppTau}, we have $g_{k(z_{i,j})}(z)>0$ for all
$j$, so $\supp \Delta'_i(x)\subset \supp g(z)$.
Consequently, we can define a map
$\bar{\Delta}_i\co \Delta^m\times [0,1]\to \Sigma$ by
$$\bar{\Delta}_i(x,t)=t g(\Delta_i(x))+(1-t) \Delta'_i(x).$$

Let $\gamma=\sum_i [\bar{\Delta}_i]$, where $[\bar{\Delta}_i]$ is the
image of the fundamental class of $\Delta^m\times [0,1]$.  Then
$\partial \gamma=g_\sharp(\alpha)-\alpha'$ as desired.  Furthermore,
since $\Lip \Delta_i\lesssim \epsilon$ and
$\Lip \Delta'_i\lesssim \epsilon$, we have
$\Lip \bar{\Delta}_i\lesssim \epsilon$, and
$$\mass \gamma \lesssim N\epsilon^{m+1}\lesssim c_\alpha \epsilon.$$

To construct $\lambda$, we use the Lipschitz connectivity of $Z$.  For
each $i$, we have $d(\Delta_i, h\circ \Delta'_i)\lesssim \epsilon$,
$\Lip(\Delta_i)\lesssim \epsilon$, and
$\Lip(h\circ \Delta'_i)\lesssim \epsilon$, so we can use the Lipschitz
connectivity of $Z$ to construct prisms $p_i\co \Delta^m\times [0,1]\to
Z$ such that $p_i|_{\Delta^m\times 0}=\Delta_i$, $p_i|_{\Delta^m\times
  1}=h\circ \Delta'_i$, and $\Lip(p_i)\lesssim \epsilon$.  Let
$\lambda=\sum_i [p_i]$.  If we are
careful to match corresponding faces in neighboring simplices, then
$$\partial \lambda=\alpha-h_\sharp(\alpha')$$
and
$$\mass \lambda\lesssim N\epsilon^{m+1}\lesssim c_\alpha \epsilon$$
as desired. 
\end{proof}
\bibliographystyle{plain}
\bibliography{higherSol}
\end{document}